\documentclass[12pt,a4 paper]{article}
\usepackage[utf8]{inputenc}
\usepackage[english]{babel}
\usepackage[left=30mm,right=15mm,
    top=2cm,bottom=2cm]{geometry}
\usepackage{amsthm, amsmath, amssymb}
\usepackage[usenames]{color}
\usepackage{comment}
\usepackage{hyperref}

\newcommand{\Pn}{\mathcal{P}_n}

\newtheorem{theorem}{Theorem}
\newtheorem*{theorem*}{Theorem}
\newtheorem{corollary}{Corollary}
\newtheorem{lemma}{Lemma}
\newtheorem*{lemma*}{Lemma}

\newtheorem*{remark*}{Remark}

\let\le\leqslant \let\leq\le
\let\ge\geqslant \let\geq\ge

\DeclareMathOperator{\supp}{supp}

\newcommand{\ds}{\displaystyle}
\newcommand{\mP}{\mathcal{P}_n}

\renewcommand{\Re}{\mathop{\mathrm{Re}}}
\renewcommand{\Im}{\mathop{\mathrm{Im}}}

\renewcommand{\supp}{\mathop{\mathrm{supp}}}

\newcommand{\Jdif}{J_{K \backslash Q}}
\newcommand{\PK}{\mathcal{P}_n(K)}
\newcommand{\w}[1]{\widetilde{#1}}
\newcommand{\NN}{\mathbb{N}}
\newcommand{\CC}{\mathbb{C}}
\newcommand{\RR}{\mathbb{R}}

\usepackage[usenames,dvipsnames,svgnames]{xcolor}

\definecolor{mag}{rgb}{0.9,0,0.9}
\definecolor{red}{rgb}{0.8,0,0}
\definecolor{green}{rgb}{0,0.8,0}
\definecolor{darkgreen}{rgb}{0,0.6,0}

\def\sz#1{{\color{darkgreen}#1}}

\newcommand{\de}{\delta}   
\newcommand{\ff}{\varphi}   

\begin{document}

\title{The growth order of the optimal constants \\ in Tur\'{a}n-Er\H {o}d type inequalities in $L^q(K,\mu)$}

\author{P. Yu. Glazyrina, Yu. S. Goryacheva and Sz. Gy. R\'{e}v\'{e}sz}


\date{\today}

\maketitle

\begin{abstract}
In 1939 Tur\'{a}n raised the question about lower estimations of the maximum norm of the derivatives of a polynomial $p$ of maximum norm $1$ on the compact set $K$ of the complex plain under the normalization condition that the zeroes of $p$ in question all lie in $K$.
Tur\'{a}n studied the problem for the interval $I=[-1,1]$ and the unit disk $D$ and found that with $n := \deg p$ tending to infinity, the precise growth order of the minimal possible derivative norm (oscillation order) is $\sqrt{n}$ for $I$ and $n$ for $D$.
Er\H{o}d continued the work of Tur\'{a}n considering other domains. Finally, in 2006,
Hal\'{a}sz and R\'{e}v\'{e}sz  proved that the growth of the minimal possible maximal norm of the derivative is of order $n$ for all compact convex domains.

Although Tur\'{a}n himself gave comments about the above oscillation question in $L^q$ norms, till recently results were known only for $D$ and $I$.
Recently, we have found order $n$ lower estimations for several general classes of compact convex domains, and proved that in $L^q$ norm the oscillation order is at least $n/\log n$ for all compact convex domains. In the present paper we prove that the oscillation order is not greater than $n$ for all compact (not necessarily  convex) domains $K$ and $L^q$-norm with respect to any measure supported on more than two points on $K$.
\end{abstract}



\section{Introduction 
}\label{sec:Introduction}

Let $K$ be a compact set of the complex plane $\CC$ and  $\mu$ be a finite Borel measure on  $K$. For a polynomial $p \in \CC[z]$ and a parameter $0<q<\infty$ we set
$$
\| p \|_q :=\| p \|_{L^q(K,\mu)} :=
\left(\int\limits_{K}|p(z)|^q \mu(dz)\right)^{1/q},
$$
and for $q=\infty$ we will also consider the limiting case\footnote{The last equality follows by continuity of $p(z)$; otherwise one should have taken
$\inf\left\{ \sup_{z\in Q} |p(z)| \colon Q\subset K,\ \mu(K\setminus Q)=0\right\}$.

Note that this definition of the $\infty$-norm is essentially independent of the measure $\mu$ (apart form its support), and hence is different form the usual weighted maximum norm definition $\|p\|_{w,\infty}:=\sup\|pw\|_\infty$.
While in case of an absolutely continuous measure $\mu$ with density function $w$ (with respect to Lebesgue or area measure $\lambda$ restricted to $K$), the weighted $L^q(K,\mu)$ norm matches the weighted norm $\|p\|_{w,q}:=\left( \int_K |p(z)|^q w(z) d\lambda(z)\right)^{1/q}$ for any $0<q<\infty$ (and analogously if the measure $\mu$ is absolutely continuous with respect to the arc length measure $\ell$ of $\partial K$), the limiting case of this relation provides our above definition, and does not always match the $\|\cdot\|_{w,q}$-norm. Therefore, such weighted maximum norms require separate studies. They may, however, be obtained as a limiting case of $L^q$ norms with respect to varying weights $w^q$, as then $\|p\|_{w,\infty}=\lim_{q \to \infty} \|pw\|_q$.
}
$$
\| p \|_\infty :=\| p \|_{L^\infty(K,\mu)}
:=\lim_{q \to \infty} \| p \|_q=\max_{z\in \supp \mu} |p(z)|.
$$
Let $d $ denote  the diameter of $K$ and $ w$ denote the width of $K$,
$$ d:=d(K) := \max_{z, z' \in K}|z - z'| ,  \quad \ds w:=w(K) := \min_{\gamma \in [-\pi, \pi]}
\max_{z,z' \in K}\left(\Re(ze^{i\gamma})-\Re(z'e^{i\gamma})\right) .
$$
The width $w$ is equal to the smallest distance between two parallel lines between which $K$ lies, hence $w\le d$. Denote by $\mP(K)$  the set of algebraic polynomials $p$ of degree exactly $n$,  all of whose zeros lie in  $K$. The (normalized) quantity under our study is the ``inverse Markov factor'' or ``oscillation factor''
\begin{equation}\label{mainconst}
M_{n,q}(K,\mu) := \inf_{\substack{p\in\mathcal{P}_n(K),\\ \|p\|_q\neq 0}} \frac{\|p'\|_q}{\|p\|_q}.
\end{equation}
The most important case is when $K$ is convex and $\mu$ is the arc length measure on the boundary of $K$, in which case we will simply write  $M_{n,q}(K)$.

Problem \eqref{mainconst} goes back to Tur\'{a}n \cite{Tur}, who
raised the question of the inequality
\begin{equation}\label{tur}
\|p'\| _{C(K)} \ge M_{n,\infty} (K)\|p\|_{C(K)}, \quad   p \in \mathcal{P}_n(K).
\end{equation}
This inequality is a kind of converse to the classical inequalities of Markov and Bernstein.
With any positive constant in place of $M_{n,q}(K)$ the inequality \eqref{tur} fails to hold on the set of all polynomials of degree $n$, as is shown by the example of polynomials $z^n+c$ as $c\to \infty$.
For this reason, we need to restrict the class of polynomials for getting a sound inequality with nonzero constant.
Tur\'{a}n imposed the additional condition $\ds p \in \mathcal{P}_n(K)$, and
studied the problem for the interval $K=I:=[-1,1]$ and the unit disk $K=D:=\{z\in \CC:|z|\leq1\}$.
He proved that
\begin{equation}\label{T1}
M_{n,\infty} (D)=\frac{n}{2}, \quad \text{and} \quad  M_{n,\infty} (I)\geq \frac{\sqrt{n}}{6}.
\end{equation}
Moreover, he pointed out by example of $(1 - x^2)^n$ that the $\sqrt{n}$ order in the latter relation cannot be improved upon. In the same year, Er\H{o}d \cite{Erod} continued Tur\'{a}n's research and showed that $M_{n,\infty}(I) = \sqrt{n/e}+O(1/n),$ $n\to \infty$. More importantly, he extended the study of Tur\'{a}n's problem to general convex domains of $\CC$, and obtained several results in the maximum norm for various general classes of convex domains. On account of the above, we will term these type of inequalities ``Tur\'an-Er\H{o}d type inequalities'', and call the respective optimal constants $M_{n,q}(K,\mu)$ the ``Tur\'an-Er\H{o}d constants'', too. The growth order of $M_{n,q}(K,\mu)$ in terms of the degree $n$ is in the focus of our study.

In the full generality of all compact convex sets Levenberg and Poletsky proved the general inequality $M_{n,\infty} (K)\geq \dfrac{\sqrt{n}}{20 d} $ \cite[Theorem 3.2]{LP}. The order of this result is best possible for the interval case, but for domains with nonempty interior a susbtantial improvement is possible. Namely, for any compact convex domain $K$
it was obtained in 2006 \cite[Theorem 1]{SzR} by Hal\'{a}sz and R\'{e}v\'{e}sz that for all $n\in \NN$
$$\ds M_{n,\infty} (K)\geq 0.0003 \frac{w}{d^2}  n,$$
and R\'{e}v\'{e}sz also proved \cite[Theorem 2]{SzR}
\begin{equation}\label{Revesz01}
M_{n,\infty}(K) \leq 600 \frac{w}{d^2} n,
\quad \textrm{for} \ n>\dfrac{d^2}{128w^2}\ln\dfrac{d}{16w}.
\end{equation}
Note that both estimates have the same form and depend only on $n$, $w$ and $d$, so that up to an absolute constant factor, even the dependence on the geometric features of the general convex body is established.

Recently these results were partially extended to the $L^q$ norms ($q\ge 1$) with respect to the arc length measure $\ell$ on the boundary curve $\Gamma$ of $K$ by Glazyrina and R\'{e}v\'{e}sz.
They proved \cite{GR1, GR2} that the growth order of $M_{n,q}(K)$ is again $n$ for a certain class of compact convex domains, including all smooth\footnote{A convex domain is smooth if it has a unique supporting line at each boundary point.}, compact convex domains $K$ and also convex polygonal domains having no acute angles at their vertices. It was conjectured \cite[Conjecture 1]{GR2} that even for arbitrary compact convex domains the growth order of $M_{n,q}(K)$ should be $n$. In \cite{GR3}, it is shown that in $L^q$ norm the oscillation order is at least $n/\log n$ for all compact convex domains.
From the other direction, Glazyrina and R\'{e}v\'{e}sz  proved that one cannot expect more than order $n$ growth because in fact $M_{n,q} (K)\le \dfrac{15}{d}n$, $(q\ge 1)$ (a combination of Theorem 5 and Remark 6 of \cite{GR1}.)

In this paper we study $L^q$-norms for finite $q$, and obtain $\infty$-norm estimates only via direct limiting cases. 
Therefore, in this introduction we also restrict mainly to $L^q$ results. Detailed overviews of further results in Tur\'{a}n type inequalities can be seen in \cite{SzR,GR1,GR2, GR3}. Note also that until our work \emph{weighted} $L^q$ norms were only considered for the interval $I$ (and even there only with absolutely continuous measures with some density function $w$). One notable result is due to Varma \cite{Varma1, Varma2}, who proved that for the interval $\sqrt{\frac12 n + \frac34 +\frac{3}{4n}} <M_{n,2}(I) \le \sqrt{\frac12 n + \frac34 +\frac{3}{4(n-1)}}$, which in itself is not a weighted result, but Varma also compared some weighted norms of $p'$ to non-weighted norm of $p$, which implied his above mentioned results.
Varma and Underhill \cite{Underhill} studied the inequality
\begin{equation}\label{Iweight}
\|p'w\|_{L^q[-1,1]} \geq C_{n,q,w}\|pw\|_{L^q[-1,1]}, \quad
p \in \mP([-1,1]).
\end{equation}
They  found the sharp value $M_{n,q}(I,w)$ of $C_{n,q,w}$
for even $n$, $q=2$, $w(x)=(1-x^2)^{\alpha}$, $\alpha>1$,
and for $q=4$, $w(x)=(1-x^2)^3$. For the cases of $q=4$, $w(x)=(1-x^2)$ and
$w(x)=(1-x^2)^2$, they established the right order of magnitude of the respective Tur\'an-Er\H{o}d constants.
Xiao and Zhou \cite[p. 198, Theorem 1, Corollary 1]{Xiao} proved that $M_{n,\infty}(I,w) \ge C(w) \sqrt{n}$
for any nonnegative, continuous, piecewise monotonic weight $w(x)$ on $[-1,1]$. They also pointed out that
for the Jacobi weight $(1+x)^{\alpha}(1-x)^{\beta}$, $\alpha,\beta \ge 0$ the $\sqrt{n}$ order cannot be improved.
Wang and Zhou  \cite[Theorems 1 and 2]{Wang} proved $M_{n,q}(I,w) \ge C(w,q) \sqrt{n}$ for generalized Jacobi weights, i.e. for weights with finite total mass $\int_{-1}^1 w(x)dx$ and satisfying that values of $w(x_1), w(x_2)$ are within constant ratio if the variables $x_1, x_2$ have bounded proportion of distances from the interval endpoints 1 or $-1$.
Yu and Wei \cite[Corollaries 1 and 2]{Yu} obtained Tur\'an type inequalities for doubling, and in case of $q=\infty$, for so-called $A^*$ weights. Subsequently, in \cite{Chinese} the results for doubling weights were extended to a somewhat larger class of weights, called "$N$-doubling weights".

Some $L^q$ results were obtained for the disk and its perimeter with the (unweighted) arc length measure on it, but truly weighted versions were not derived. Combining the results of Malik \cite{Malik69} (obtained for $R \le 1$) and Govil \cite{Govil} (proved for $R \ge 1$), it is known that denoting $D_R:=\{z\in\CC~:~|z| \le R\}$ and
putting $M_{n,q}(K,Q):=\inf_{P \in \mP(K)} \| P'\|_{L^q(Q)} /\|P\|_{L^q(Q)}$, we have
$$
M_{n,\infty}(D_R,D)=
\begin{cases}
\dfrac{n}{1+R}, R\le 1 \\ \dfrac{n}{1+R^n}, R\ge 1
\end{cases}
.
$$
Regarding maximum norm, there are some further results in the literature for the two discs case, sharpening the above in case there is suitable information on zeroes or coefficients, but these results do not improve the above--sharp in themselves--inequalities in general.

In \cite{Malik} the exact Tur\'an-Er\H{o}d-type comparison constant of the circle is computed between $\infty$-norm and $q$-norm, i.e. it is proved that for any $q>0$ and $P\in P_n(D)$ it holds $\|p'\|_{\infty} \ge A_q \|p\|_q$, where $A_q$ is an explicit constant, attained in case $P(z)=z^n+1$ (and equivalent polynomials). This was generalized by Aziz \cite{Aziz} to any $L^q$ norms, even different ones on the two sides of the respective inequalities, both for
zeroes in $D_R$ with $R\le 1$ and also for $R>1$.

Formerly \cite{GR1, GR2, GR3}, we extended $L^q$-norm investigations from the interval and circle case to boundary arc length $L^q$ norm of compact convex sets, but only here we proceed towards general weighted norms, moreover even to measures $\mu$ and $L^q(K,\mu)$ norms.

This seems to be appropriate also regarding the other, natural possibility of extensions, namely, regarding the ``two-set problem'' with two prescribed sets: one, say $Q$, for taking the norm and one, say $K$, for the location of zeros. Maximum norm results for two sets $K$ and $Q$ can be viewed as the limiting case of $L^q(K,\mu)$ estimates with respect to a fixed measure $\mu$ supported exactly in $Q$, while restricting the zeros of the considered polynomials to lie in $K$ (i.e., the polynomials to belong to $\mP(K)$). As discussed above, that type of results were derived in the disk case with $Q$ being the unit circle and $K$ being another concentric disk \cite{Malik69, Govil, Aziz}.

As an enlightening example (and the only one we know about apart from the concentric circle cases), let us quote a recent result of Komarov. In 2019, Komarov \cite{Komarov} obtained the estimate $M_{n,\infty}(D^+,I)\ge 2\sqrt{n} /(3\sqrt{210e})$ for  $K=D^{+}:=\{z\in D~:~ \Im z\geq0\}$, where in general we mean $M_{n,\infty}(K,Q):=\inf_{p\in \mP (K)} \|p'\|_{C(Q)}/\|p\|_{C(Q)}$. The $\sqrt{n}$ order in this estimate is sharp\footnote{Indeed, already Tur\'an showed that there exists a polynomial $p\in \Pn(I)$ with $\|p'\|_{C(I))}=O(\sqrt{n}) \|p\|_{C(I)}$, and any such polynomial is automatically in Komarov's class $\Pn(D^{+})$.}.  As explained above, this can be considered as a special case of a measure defined weighted norm estimate, with $\supp \mu=I$ (say $\mu$ can be the linear Lebesgue measure on $I$).

We also note that Erd\'{e}lyi \cite{Erd} generalized Komarov's result further to the set
of polynomials of degree $n$ having at least $n-\ell$ zeros in $D^+$ and at least one zero in $[-1,1].$

As Komarov's result shows, if the width of $\supp \mu$ is zero, then the result may follow the pattern of Tur\'{a}n's result for $I$, the order of growth being $\sqrt{n}$. We are more interested in the second case, when the support of the measure does not degenerate to a subset of a straight line segment. The only cases known are cases when $\mu$ is the arc length measure $\ell$ of the boundary curve $\partial K$ of a certain compact convex domain $K$; and in all known cases the growth order of $M_{n,q}(K)$ is $n$. However, for fully general compact convex domains we only know estimates to the effect $n/\log n \ll M_{n,q}(K) \ll n$ and only conjectured that for all compact convex domains $K$ the order of $M_{n,q}(K)$ should indeed be $n$. Here we will extend investigations to other measures, and will in particular consider the situation with $\mu$ being the area- (Lebesgue-) measure $\lambda$ restricted to $K$, where $K$ is a compact convex domain.

Let us recall the following. From Tur\'{a}n's proof \cite[the footnote on p. 93]{Tur} it follows that if for the point $z\in \partial K$ there is a disk $D_R$ of radius $R$
passing through $z$ and containing $K$ then
\begin{equation}\label{Rcirc}
|p'(z)|\ge \frac{n}{2R}|p(z)|.
\end{equation}
In 2002, Levenberg and Poletsky [31, Proposition 2.1] called a set $K$ $R$-circular
if  \eqref{Rcirc} holds for all points $z\in \partial K$
\cite[p. 176, Theorem 2.2]{LP}.
 It is easy to verify that an $R$-circular set is always convex
\cite[p. 176, Corollary 2.3, Remark]{LP}.
Suppose that  $K$ is $R$-circular and a finite measure $\mu$ is supported on $\partial K$ -- the boundary of $K$.
Raising inequality \eqref{Rcirc} to a power $q>0$ and integrating over $\partial K$
with respect to  $\mu$, we get
\begin{equation*}\label{Rcirc1}
\int_{\partial K}|p'(z)|^q \mu(dz) \ge  \frac{n^q}{(2R)^q} \int_{\partial K}|p(z)|^q \mu(dz).
\end{equation*}
Hence  $M_{n,q}(K,\mu)\ge \dfrac{n}{2R}$. This was already noted by Tur\'{a}n for the disk (and arc length on $\partial D$), and formally presented by Levenberg and Poletsky for general $R$-circular domains and the boundary arc length measure $\ell$ on $\partial K$, see Proposition 5 in \cite{LP}. However, it automatically extends to any measure supported on $\partial K$ as well.

Our  goal is to obtain an upper bound for $M_{n,q}(K,\mu)$ in terms of  $n$,  $d$,  $w$, and some characteristics of the measure $\mu$. Let us note that we cannot expect any sound results if $\supp \mu$ is finite, for one can easily guarantee (for sufficiently large degree $n$) that $p'$ vanishes on a given prescribed finite set in $K$. Therefore, we can surely restrict our attention to the case when for two points $A$ and $B$ the measure is not supported on just those two points. In other words, $\mu(\{A,B\})<\mu(K)$, i.e. $\mu(K\setminus\{A,B\})>0$. Let us select and fix any diameter of $K$, and let $A$ and $B$ be the endpoints of this given diameter. Then the straight lines $a$ and $b$, passing through $A$ and $B$ and orthogonal to the diameter $AB$, must be strict supporting lines, so that $a\cap K=\{A\}$, and $b \cap K =\{B\}$ only. In fact, more is true: K is a subset of the disks of radius $d$ about $A$ and about $B$ as well. Therefore, outer regularity of the measure $\mu$ entails that with sufficiently close parallel lines $a'$ and $b'$ to $a$ resp. $b$, even the $\mu$-measure of the part of $K$ between $a'$ and $b'$ stays positive. These considerations motivate our formal parameterization of the main result of the paper.

For the following, we introduce some further notation.
Let $K$ be a compact set, and points $A, B\in K$ be such that $|B-A|=d$. We will denote by $K_\delta$ the part of the set $K$, enclosed between two parallel lines that are perpendicular to the segment $[A,B]$ and located at the distance $\delta d/2$ from the midpoint $(A+B)/2$ of the segment, i.e.,
\begin{equation}\label{Kdelta}
K_\delta:=\left\{z\in K :  \Re\right((z-(A+B)/2)e^{-i\arg(B-A)}\left) \in [- \delta d/2,  \delta d/2] \right\}.
\end{equation}

\begin{theorem}\label{theorem1}
Let $\ds K $ be a compact subset of~$\mathbb{C}$ with  width $w>0$ and diameter $d$. Suppose that a finite non-negative measure $\mu$
is given on $K$ and there are $\theta \in (0,1)$ and $\delta \in (0,1)$ such that
\begin{equation}\label{ogrMu}
\mu\left(K_\delta\right)
\geq \theta \mu(K).
\end{equation}
Then for any $0<q\le \infty$ and
\begin{equation}\label{ncondTh}
\ds n \ge 2(1+1/q)\frac{d^2}{w^2}\ln\frac{d}{w}
\end{equation}
we have\footnote{We set $1/q=0$ for $q=\infty$.}
\begin{equation}\label{Main}
M_{n,q}(K,\mu)\le
C_q(\delta,\theta)\frac{w}{d^2}n, \quad \mbox{where} \quad
C_q(\delta,\theta):= \dfrac{121}{1-\delta}\left(1 +\frac{2}{\theta}  \right)^{1/q}.
\end{equation}
\end{theorem}

If $K$ is convex and  $\mu$ is the area on $K$ (i.e., the restriction to $K$ of the Lebesgue measure $\lambda$ on the plain) or $\mu$ is the (linear) Lebesgue measure (i.e. the arc length measure $\ell$) on the boundary of $K$, then \eqref{ogrMu} take place for any $\theta \in (0,1)$ with a suitably chosen $\delta$. As we  will show in the last section, it is possible to estimate  $\theta$ via $\delta$, and then optimize with respect to $\de\in(0,1)$, finally getting a bound in function of $q$ only. This leads to the following corollaries.

\begin{corollary}\label{cor1}
Let $K\subset \CC$ be a compact convex subset of $\CC$ having width $w>0$ and diameter $d$.
Let $0<q\le \infty$, and
$\ds {n \ge 2(1+1/q)\frac{d^2}{w^2}\ln\frac{d}{w}}$.
If $\mu$ is  the linear Lebesgue measure on the boundary of $K$ (arc length measure $d\ell$),  then we have
$$
M_{n,q}(K) \le C_q \frac{w^2}{d} n,
$$
where
\begin{equation}\label{C1}
 C_q:= 121 \frac{3q+2+2\sqrt{q^2+3q+1}}{5q} \left(3+2q+2\sqrt{q^2+3q+1}\right)^{1/q}.
\end{equation}
\end{corollary}
As a consequence, if $q=\infty$, then $M_{n,\infty}(K)\le 121 \dfrac{w^2}{d} n$ for $n\ge 2\dfrac{d^2}{w^2}\ln \dfrac d w$, improving upon the earlier estimate \eqref{Revesz01}.

\begin{corollary}\label{cor2}
Let $K\subset \CC$ be a compact convex subset of $\CC$ having width $w>0$ and diameter $d$.
Let $0<q\le \infty$, and
$ \ds {n \ge 2(1+1/q)\frac{d^2}{w^2}\ln\frac{d}{w}}$.
If $\mu$ is the two dimensional Lebesgue measure $\lambda$ on $K$ (area), then we have
$$
M_{n,q}(K,\lambda) \le C_q^*\frac{w^2}{d} n,
$$
where
\begin{equation}\label{C2}
C_q^*:=121 \frac{5q+4+2\sqrt{4q^2+10q+4}}{9q} \left( 4q+5+2\sqrt{4q+10q+4}\right)^{1/q}.
\end{equation}
\end{corollary}

\section{Auxiliary results for the upper estimate of $M_{n,q}(K,\mu)$ }\label{ar}

Our proof starts with the observation that the constant $C_q(\delta,\theta)$ is invariant under an affine map of $\CC$. To be more precise, consider an affine map
$\ds \phi(z) = \kappa (z- z_0)$, $\kappa\in \CC\setminus \{0\},$ $z_0\in \CC$,
and the image $
\widetilde{K}=\{\phi(z) \  : \ z\in K\}$ of $K$ under this map.
Define the Borel measure $\w{\mu}$ on $\w{K}$ by the formula $\w{\mu} (\widetilde{E})= \mu ( \psi(\widetilde{E})),$ where $\psi(t)=t/\kappa +z_0$ is the inverse map to $\phi.$

The widths and the diameters of $K$ and $\widetilde{K}$ are related  by equalities
\begin{equation*}\label{dwK}
\ds w(\widetilde{K}) = \kappa w(K), \qquad d(\widetilde{K})=\kappa d(K).
\end{equation*}
The parameters $\theta$ and $\delta$ in \eqref{Kdelta}
are the same for $K$ and $\widetilde{K}.$

A polynomial $p(z)\in \PK$ if and only if  $\widetilde{p}(t):=p(\psi(t)) \in \mathcal{P}_n(\widetilde{K})$. By \cite[p.~190, Theorem~3.6.1]{Bogachev} we have
$$
\|\w p\|_{L^q(\w K)}^q=\int_{\w K}|\w p(t)|^q (\mu \circ \psi) (dt)=\int_{K}|\w p(\phi(z))|^q \mu(dz)
=\int_{K}|p(z)|^q \mu(dz)= \|p\|_{L^q(K)}^q,
$$
and $ \|\w{p} '\|_{L^q(\w K)}=(1/\kappa)\|p'\|_{L^q(K)}$,
as  $\ds \w{p}'(t) = p'(\psi(t))\psi'(t) =(1/\kappa)p'(\psi(t)).$

If we will obtain  estimate~\eqref{Main} for the set $\w{K}$, then
that will yield
$$
\frac{\|p'\|_q}{\|p\|_q}=\kappa\frac{\|\widetilde{p}'\|_q}{\|\widetilde{p}\|_q}\le \kappa C_q(\delta,\theta) \frac{w(\w{K})}{d^2(\w{K})}n=C_q(\delta,\theta)\frac{w(K)}{d^2(K)}n,
$$
i.e. we will obtain estimates for $K$ and $\mu$ with the same constant factor.

On account of the above observation, we can assume without loss of generality that the compact set $K$ has diameter $d=2$ and the selected diameter endpoints are the points $-1$ and $1$.
In this case the set $K_\delta$ defined by \eqref{Kdelta} takes  the form
\begin{equation}\label{Kdeltanorm}
K_\delta=\{x+iy\in K : x\in[-\delta,\delta]\}.
\end{equation}

\begin{lemma}\label{lemma1}
Let $\ds K $ be a set having width $w>0$ and diameter $d=2$, with the points $-1$ and $1$ belonging to $K$. Then $K$ is contained in the rectangle $[-1,1]\times [-iw,iw]$.
\end{lemma}

\begin{proof}
The  idea of the proof was proposed by  Hungarian mathematician S\'{a}ndor Krenedits.

Denote by $c^*=x+iy^*$ a point on the boundary of $K$ such that
$\ds |y^*| = \max_{x+iy \in K} |y|.$
Obviously, $K$ is contained in the  rectangle $[-1,1]\times [-i|y^*|,i|y^*|]$.
We need to show that $|y^*|\le w.$
Denote by $T$ the triangle with the vertices $a=-1$, $c^*$, $b= 1$.
As $w$ is the width of $K$ and all vertices of $T$ are points of $K$, we have $w(T)\le w=w(K)$.
The width of $T$ is equal to the smallest height of the triangle,
i.e.  the height drawn to the longest side.
Since the largest side is $ab$ ($2$ is the diameter of $K$),
the length of the smallest height equals $|y^*|$.
Hence $|y^*|=w(T) \le w(K).$
\end{proof}

\begin{lemma}\label{lemma2}
Let $ K $ be a compact set  with a positive width $w>0$ and diameter $d=2$,
and let $ K $ contain the points $-1$, $1.$
Suppose that a finite non-negative measure $\mu$
is given on $K$ and there are $\theta \in (0,1)$ and $\delta \in (0,1)$ such that
\eqref{ogrMu} holds.
If $w+\delta<1$, then
there exists an interval $[A-w,A]$ such that either $0<A\le \delta$ or
$-\delta \le A-w<0$ and
the set  $$Q^* =  \{x+iy \in K:\, A-w \leq x \le A \}$$
satisfies the inequality
\begin{equation}\label{lemma2_02}
\mu(Q^*) \geq \frac{w}{2}\theta\mu(K).
 \end{equation}
\end{lemma}

\begin{proof}
Let us  set $L:=\lceil\delta/w\rceil$ (the least integer greater than or equal to $\delta/w$). Note that $L\ge 1$ by definition. Let us represent the set $K_\delta$ as $$K_\delta=\bigcup\limits_{\ell=-L+1}^L Q_\ell,
\ \mbox{where} \
Q_{\ell} :=\{x+iy\in K:  x\in \left[(\ell-1) w, \ell w\right]\cap [-\delta,\delta]\}.
$$
Take $\ell_0 \in\{-L+1,\ldots,L\}$  such that $\ds \mu(Q_{\ell_0})=\max_\ell \mu(Q_{\ell})$, then
$$
\theta \mu(K)\le \mu(K_\delta)\le \sum_{\ell=-L+1}^L \mu(Q_{\ell})
\le 2L \mu(Q_{\ell_0})\le 2\dfrac{\delta+w}{w} \mu(Q_{\ell_0})\le \dfrac2w \mu(Q_{\ell_0}),
$$
hence $\mu(Q_{\ell_0})\ge \dfrac w2 \theta \mu(K)$.

If $\ell_0>0$ we set $A:=\min \{\ell_0 w, \delta\}$, and
if $\ell_0\le 0$ we set $A:=w+\max \{-\ell_0w, -\delta\}$.
In both cases $Q_{\ell_0} \subset Q^*$, which implies  \eqref{lemma2_02}.
\end{proof}

The proof of Theorem~\ref{theorem1} will consist of a construction of a suitable polynomial $p$ and careful estimates of its norm and the norm of its derivative. We will take this polynomial in the form
\begin{equation}\label{pnk}
p(z)=(1+z)^{n-k}(1-z)^k.
\end{equation}
It is worth pointing out that  the values of $|p(z)|=|p(x+iy)|$
increase with increasing  $|y|$.
For this reason we will need estimates $|p|$ on the intervals $[-1,1]$ and $[-1+iw,1+iw].$
In the following three lemmas we study the behaviour of $|p|$ on these intervals.

\begin{lemma}\label{lemma4}
Suppose that $n$ and $k$ are positive integers, $2k <n$.
Then $p(x)=(1+x)^{n-k}(1-x)^k$ attains its maximum on $[-1,1]$ at the point $M=1-2k/n \in (0,1)$,
increases on $[-1,M]$, decreases on $[M,1]$, and
\begin{equation}\label{compp}
p(M-x) \ge p(M+x) \quad \mbox{for all} \quad x\in [0, 1-M].
\end{equation}

\end{lemma}

\begin{proof}
An easy computation shows that
\begin{equation}\label{pp}
p'(x) = n(1+x)^{n-k-1}(1-x)^{k-1}(1-2k/n-x),
\end{equation}
thus the polynomial $p$  attains its maximum on $[-1,1]$
at the point $M=1-2k/n \in (0,1)$, increases on $[-1,M]$, decreases on $[M,1]$.
We can write  $p(M+x)$ for $x\in(-(1-M),1-M)$  in the form
\begin{align*}
p(M+x)&=(1+M)^{n-k}(1-M)^{k}\left(1+\frac{x}{1+M}\right)^{n-k}\left(1-\frac{x}{1-M}\right)^{k}
\end{align*}
and
$$
\ln p(M+x) = \ln p(M)+n\tau(x), \ \text{where} \
\tau(x) = \left(1-\frac{k}{n}\right)\ln\left(1+\frac{x}{1+M}\right)+\frac{k}{n}\ln\left(1-\frac{x}{1-M}\right).
$$
We note that $k/n = (1-M)/2$, $1-k/n=(1+M)/2$ and expand $\tau(x)$
in a Taylor series on the interval $(-(1-M), 1-M)$:
\begin{align*}
\tau(x)&=\frac{1+M}{2}\sum_{\ell=1}^{\infty}\frac{(-1)^{\ell-1}x^{\ell}}{\ell(1+M)^{\ell}}+
\frac{1-M}{2}\sum_{\ell=1}^{\infty}\frac{(-1)^{\ell-1}(-1)^{\ell}x^{\ell}}{\ell(1-M)^{\ell}}\\
&=
\sum_{\ell=1}^{\infty}\frac{(-1)^{\ell-1}(1-M)^{\ell-1}-(1+M)^{\ell-1}}
{2\ell(1-M^2)^{\ell-1}}x^{\ell}.
\end{align*}
Since  $1-M^2>0$, it follows that for odd $\ell$ the sign of the coefficient  of $x^\ell$
equals the sign of the expression
$\ds (-1)^{\ell-1}(1-M)^{\ell-1} -(1+M)^{\ell-1}=(1-M)^{\ell-1}-(1+M)^{\ell-1}<0$.
This proves that $\tau(x)<\tau(-x),$ $x\in[0, 1-M)$, and consequently  $p(M-x) \ge p(M+x)$ for $x\in[0, 1-M)$.  The inequality is valid  at the point $x=1-M$ due to the continuity of $p$.
\end{proof}

For given positive integers $n$ and $k$, $0<2k\le n$ and a given $w\in(0,1)$ we introduce
the following notation
\begin{equation*}\label{fnk}
f_{n,k}(x)=((1+x)^2+w^2)^{n-k}((1-x)^2+w^2)^k, \quad M_{n,k}:=1-2k/n.
\end{equation*}

\begin{lemma}\label{lemma5}
Suppose that $n$ and $k$ are positive integers, $0<2k \sz{<}  n$, $M=M_{n,k},$
$w\in(0,(1-M)/4)$.
Then there are points
$$\ell_{n,k}\in(-1, -M), \quad  \widetilde{M}=\widetilde{M}_{n,k}\in (M,\, M+w/2), \quad \mbox{and} \quad r_{n,k}\in (1-w/2,1)$$
such that $f_{n,k}$ decreases on $[-1,\ell_{n,k}]$ and $[\widetilde{M},r_{n,k}]$, and increases on $[\ell_{n,k},\widetilde{M}]$ and $[r_{n,k},1]$.
\end{lemma}

\begin{proof}
As observed above  $k/n=(1-M)/2$ and $(n-k)/n=(1+M)/2$.
Our goal is to estimate  extremum points of the function
$$
h(x)=\frac1n \ln f_{n,k}(x)=
\frac{1+M}{2}\ln((1+x)^2+w^2)+\frac{1-M}{2}\ln((1-x)^2+w^2).
$$
We have
$$h'(x)  = \frac{(1+M)(1+x)}{(1+x)^2+w^2}-
 \frac{(1-M)(1-x)}{(1-x)^2+w^2}=\frac{2u(x)}{((1+x)^2+w^2)((1-x)^2+w^2)},
$$ where
$u(x)         =(M+x)w^2+(M-x)(1-x^2)=x^3-Mx^2-(1-w^2)x+M(1+w^2).$
By Descartes' rule of signs, the polynomial $u(x)$ has two positive zeros or no positive zeros at all.
It is easily seen that
$u(-1)<0$ and $u(-M), \, u(M),\, u(1)>0$.

Suppose $x\in (M,1),$ then $M-x<0$ and
\begin{align*}
u(x)&\le (M+1)w^2+(M-x)(1-x)(1+M)=(1+M)(x^2-x(1+M)+M+w^2).
\end{align*}
The parabola $v(x)=x^2-x(1+M)+M+w^2$ vanishes at the points\footnote{Here we use $4w <1-M$.}
$$x_1=\frac{1}{2}\left(1+M-\sqrt{(1-M)^2-4w^2} \right)
\ \mbox{and} \ x_2=\frac{1}{2}\left(1+M+\sqrt{(1-M)^2-4w^2} \right).$$
We can  estimate the left zero from above and the right zero from below  as
\begin{align*}
x_1&=\frac{1}{2}\left(1+M-(1-M)+(1-M)\left(1-\sqrt{1-4w^2/(1-M)^2}\right) \right)=\\
&=
M+\frac{1-M}{2}\frac{1-(1-4w^2/(1-M)^2)}{1+\sqrt{1-4w^2/(1-M)^2}}\le
M+\frac{1-M}{2} \frac{4w^2}{(1-M)^2}
\\&=
M+ \frac{2w^2}{1-M}\le M+\frac{w}{2},
\end{align*}
$$
x_2=(1+M)/2+ (1+M)/2-x_1\ge 1+M-M-w/2=1-w/2.
$$
Hence $u(x)$ lying below $v(x)$, is negative on $[M+w/2, 1-w/2]$.
Therefore, $u(x)$ has a zero at some point in $(M,\, M + w/2)$ -- we denote this point by $\widetilde{M}$ -- and another zero at some point in $(1-w/2,\,1)$, which we denote by $r_{n,k}$.

Observe that $u(x)$ changes its sign on $[-1,-M]$, $u(x)$ has at most three zeros
and $u(x)$ has two zeros on $(M,1)$. It follows that $u(x)$ has one zero
on $(-1,-M)$, which we denote by $\ell_{n,k}$, the other two which we denoted by $\widetilde{M}$ and $r_{n,k}$, and no more.
The Lemma is proved.
\end{proof}

\begin{lemma}\label{lemma6}
Let $w\in(0,1)$ and $A\in(0,1)$
be given numbers such that  $B:=A+3w<1$.
Suppose that positive integers $n$ and $k$ are chosen such that $0<2k < n-1$,
\begin{equation}\label{nw}
n\ge 4/w  \quad (\Leftrightarrow 2/n\le w/2),
\end{equation}
and
\begin{equation}\label{Mcond}
A+w \le M=1-2k/n \le A+\frac32 w.
\end{equation}
If
\begin{equation}\label{l503}
w\le \dfrac{1-B}{\sqrt{e^2-1}+2},
\end{equation}
then
\begin{equation}\label{fn1k}
f_{n-1,k}(x)\le  f_{n-1,k}(A-2w), \quad x\in[-1,A-2w],
\end{equation}
and
\begin{equation}\label{fn1k1}
 f_{n-1,k-1}(x)\le  f_{n-1,k-1}(B+2w), \quad x\in[B+2w,1].
\end{equation}
\end{lemma}

\begin{proof}
Let us first verify \eqref{fn1k}.
To simplify notation we set $M_1=M_{n-1,k}=1-\dfrac{2k}{n-1}$.
Since $2k<n-1$ by condition, we get
$$
0<M-M_1=\dfrac{2k}{n-1}-\dfrac{2k}{n}=\frac{2k}{n(n-1)}\le \dfrac{w}{4}  \quad
\mbox {and}\quad A- 2w< M - \frac12 w\le  M_1.
$$
By Lemma \ref{lemma5}\footnote{ We need here $2k<n-1$, not just $2k\le n-1$.} there are points $\ell_{n-1,k}\in (-1,-M_1)$
and $\widetilde{M}_{n-1,k}\in(M_1,M_1+w/2)$
such that the function $f_{n-1,k}$ decreases on $[-1, \ell_{n-1,k}]$
and increases on $[\ell_{n-1,k}, \widetilde{M}_{n-1,k}]$.
Thus it is enough to verify $f_{n-1,k}(-1)\le  f_{n-1,k}(A-2w).$
Applying inequality $\ln x\le x-1,$ $x>0$, gives
\begin{equation*}
\begin{aligned}   &\frac{1}{2}\ln\frac{4+w^2}{(1-A+2w)^2+w^2}\le    \ln\frac{2}{1-A+2w}\le
\frac{2}{1-A+2w}-1=  \frac{1+(A-2w)}{1-(A-2w)}    \le  \frac{1+M_1}{1-M_1}.
\end{aligned}
\end{equation*}
Using the last estimate we deduce
\begin{align*}
&\frac1n\ln\frac{ f_{n-1,k}(-1)}{ f_{n-1,k}(A-2w)}=
\frac{1+M_1}{2}\ln\frac{w^2}{(1+A-2w)^2+w^2}
+\frac{1-M_1}{2}\ln\frac{4+w^2}{(1-A+2w)^2+w^2}
\\
&\le\frac{1+M_1}{2}\ln \frac{w^2}{(1+A-2w)^2+w^2}+(1-M_1)\dfrac{1+M_1}{1-M_1}= \frac{1+M_1}{2}\ln \frac{e^2w^2}{(1+A-2w)^2+w^2}.
\end{align*}
Solving the inequality $\dfrac{e^2w^2}{(1+A-2w)^2+w^2}\le 1$, we obtain
\begin{equation}\label{l501}
\frac{e^2w^2}{(1+A-2w)^2+w^2}\le 1 \quad \Leftrightarrow \quad  \sqrt{e^2-1}w\le 1+A-2w
\quad \Leftrightarrow \quad w\le \frac{1+A}{\sqrt{e^2-1}+2}.
\end{equation}
Condition \eqref{l501} is  weaker than \eqref{l503}.

The proof for \eqref{fn1k1} is similar.
To simplify notation we set $M_2=M_{n-1,k-1}=1-\dfrac{2(k-1)}{n-1}$.
We have
$$
M_2-M=\dfrac{2k}{n}-\dfrac{2(k-1)}{n-1}=\frac{2(n-k)}{n(n-1)}\le \frac 2n\le \frac w2,
\quad M_2+w/2\le M+w<B \le B+2w.
$$
By Lemma \ref{lemma5} there are points $\widetilde{M}_{n-1,k-1}\in (M_2,M_2+w/2)$ and $r_{n-1,k-1}\in (1-w/2,1)$ such that
the function $f_{n-1,k-1}$ decreases on $[\widetilde{M}_{n-1,k-1},r_{n-1,k-1}]$
and increases on $[r_{n-1,k-1},1]$.
Thus it is enough to verify $f_{n-1,k-1}(1) \le  f_{n-1,k-1}(B+2w).$
We have
\begin{equation*}
\begin{aligned}
   &\frac{1}{2}\ln\frac{4+w^2}{(1+B+2w)^2+w^2}\le
    \ln\frac{2}{1+B+2w}\le  \frac{1-(B+2w)}{1+(B+2w)}
    \le  \frac{1-M_2}{1+M_2}
\end{aligned}
\end{equation*}
and
\begin{align*}
&\frac1{n-1}\ln\frac{ f_{n-1,k-1}(1)}{ f_{n-1,k-1}(B+2w)}= \frac{1+M_2}{2}\ln\frac{4+w^2}{(1+B+2w)^2+w^2}
\\&+\frac{1-M_2}{2}\ln\frac{w^2}{(1-B-2w)^2+w^2}
\le\frac{1-M_2}{2}\ln \frac{e^2w^2}{(1-B-2w)^2+w^2} .
\end{align*}
Solving the inequality $\ds \frac{e^2w^2}{(1-B-2w)^2+w^2}\le 1$, we find
$
\ds
w\le \frac{1-B}{\sqrt{e^2-1}+2}.
$

The Lemma is proved.
\end{proof}

\section{Outline of the proof of Theorem \ref{theorem1}}\label{sec:outline}

Recall that on account of  the observation at the beginning of Section \ref{ar}, we can  assume that $K$ has diameter $d=2$ and contains the points $-1$ and $1$.
In this case the set $K_\delta$, defined by \eqref{Kdelta}, takes  the form \eqref{Kdeltanorm}.

We put 
\begin{equation}\label{mcond}
  m:= 5+\sqrt{e^2-1}.
\end{equation} 
The proof of Theorem~\ref{theorem1} will be dived into two cases: the case when the width $w$ is relatively small, namely, $w \le (1-\delta)/m$, and the case when the width $w$ is large, $w >(1-\delta)/m$.

In both cases we first establish the assertion of Theorem~\ref{theorem1} for $0<q<\infty$, and then use a passage to the limit to deal with  $q=\infty$.

\section{The case of a small width}\label{Secsmallw}
Suppose that $$w\le (1-\delta)/m.$$
Then $w+\delta<1$ and by Lemma~\ref{lemma2}
there exists an interval $[A-w,A]$ such that either $0<A\le \delta$ or
$-\delta \le A-w<0$ and the set  $Q^* =  \{x+iy \in K:\, A-w \leq x \le A \}$
satisfies \eqref{lemma2_02}.
Without loss of generality we assume that $0<A\le \delta$.
The  proof falls  into four subsections.

\subsection{Construction of a polynomial and a partition of the set $K$}\label{Ss}


Let us take
$B:=A+3w$
 and  introduce another bigger set
$$
Q := \left\{x+iy \in K:\, A-2w \leq x \leq B+2w\right\}.
$$
Our conditions imply that
$$
A-2w\ge -2w \ge -\dfrac{2}{m}>-1, \quad B+2w=A+5w< A+mw\le A+1-\delta\le 1,
$$
thus $Q\subset[-1,1]\times[-iw,iw].$

Consider the polynomial $p(z) = (1+z)^{n-k}(1-z)^k$ for $2\le 2k <n$.
By Lemma~\ref{lemma4} it  attains its maximum on $[-1,1]$
at the point  $\ds M=1-2k/n$.

If we take $n$ subject to \eqref{nw}, then we can choose $k:=k(n)$ such that \eqref{Mcond} is satisfied. Further, our choice of $A$ guarantees that $[A+w,A+\frac32 w]\subset(0,1).$

Suppose that $0<q<\infty.$ Let us write the norms $\|p\|^q_q$ and $\|p'\|^q_q$ as
$$
\|p\|^q_q=
\int_{Q}|p(z)|^{q} \mu(dz)+\int_{K \backslash Q}|p(z)|^{q}\mu(dz)=:J_Q+\Jdif,
$$
and
$$
\|p'\|^q_q=
n^q\left(\int_{Q}\frac{|p'(z)|^q}{n^q}\mu(dz)+\int_{K\backslash Q}\frac{|p'(z)|^q}{n^q} \mu(d z)\right)=:
n^q\left(J'_Q+\Jdif'\right).
$$
In the next subsections we prove two inequalities of the form
\begin{equation}\label{ab}
J'_Q\leq \alpha w^qJ_Q
\quad  \mbox{and} \quad
\Jdif'\leq \beta w^q J_{Q}.
\end{equation}
Since
\begin{equation*}
\|p'\|_q=n (J'_Q+\Jdif')^{1/q} \leq n  (\alpha+\beta)^{1/q} w J_Q^{1/q} \le
4(\alpha+\beta)^{1/q} \dfrac{w}{d^2}n\|p\|_q,
\end{equation*}
this will allow us to derive
\begin{equation}\label{abf}
C_{q}(\delta,\theta)\le 4(\alpha+\beta)^{1/q}.
\end{equation}

In the following we will need some estimates of the same type.
It is convenient to list them all right here at the beginning of the argument. So we will need
\begin{equation}\label{e1}
1-A\ge (1-\delta)\ge mw,
\end{equation}
\begin{equation}\label{e1}
    1-B=1-A-3w\ge (m-3)w,
\end{equation}
\begin{equation}\label{e2}
 1-(B+w) \ge (m-4)w,
\end{equation}
\begin{equation}\label{e3}
  1-(B+2w) =1-A-5w \ge 1-\delta-5(1-\delta)/m=(1-\delta)(1-5/m),
\end{equation}
\begin{equation}\label{e4}
B+w-M = A+4w-M\ge ( 5/2)w\ge 2w.
\end{equation}

\subsection{Estimate of $J'_{Q}$}
Using representation  \eqref{pp} with $z$ in place of $x$,  we estimate the integral $J'_Q$ as
$$
J'_Q = \int_{Q} \left(|1+z|^{(n-k)}|1-z|^{k}\right)^q\left(\frac{|M-z|}{|1+z||1-z|}\right)^q \mu(dz) \le
\max_{z\in Q }\left(\frac{|M-z|}{|1+z||1-z|}\right)^q J_Q.
$$
It is easily seen that  ($z=x+iy$, $x,y\in \RR$)
\begin{align*}
&\frac{|M-z|^2}{|1+z|^2|1-z|^2} =
\frac{\left(M-x\right)^2+y^2}{((1+x)^2+y^2)((1-x)^2+y^2)}
\le \frac{\left(M-x\right)^2/w^2+1}{(1-x^2)^2} w^2.
\end{align*}

By definition of the set $Q$, we get
$$
(x-M)/w\le (B+2w-A-w)/w=4 \quad
\text{and} \quad (M-x)/w \le (A+1.5 w-A+2w)/w \le 3.5,
$$
thus $(M-x)^2/w^2\le 16$, whilst by \eqref{e3} we have
$$
1-x^2\ge 1-x\ge 1-(B+2w)\ge (1-\delta)(1-5/m).
$$
Hence we obtain the estimate
\begin{equation}\label{JpQ}
J'_Q \le  \dfrac{17^{q/2}}{(1-\delta)^q(1-5/m)^q}w^q J(Q).
\end{equation}

\subsection{Estimate of $J'_{K\setminus Q}$}\label{SsJpKQ}

1.) We first estimate $J_{Q}$ from below.
By Lemma~\ref{lemma4} the polynomial  $p(x)$ increases on ${[A-w, A]\subset[-1,M]}$, hence for $z\in Q^*$
\begin{equation*}
|p(z)|=|1+z|^{n-k}|1-z|^{k} = ((1+x)^2+y^2)^{(n-k)/2}((1-x)^2+y^2)^{k/2}
\geq p(x) \geq p(A-w).
\end{equation*}
It follows that
\begin{equation}\label{JQA}
J_{Q}\ge \int_{Q^*}|p(z)|^q d\mu(z)
\ge p(A-w)^q\mu(Q^*) \ge  p(A-w)^q\frac{\theta}{2}w \mu(K).
\end{equation}
Furthermore, $p(A-w)=p(M-(M-A+w))\ge p(M+(M-A+w))$  by \eqref{compp} and
$p(M+(M-A+w)) \ge p(B+w)$ as $M<M+M-A+w \le A+\frac32 w+\frac32 w+w= B+w<1$,
and $p(x)$  decreases on $[M,1]$.
Combining this estimate and \eqref{JQA}  yields
\begin{equation}\label{JQB}
J_{Q}\ge p(B+w)^q\frac{\theta}{2}w \mu(K).
\end{equation}

To  estimate $J'_{K\setminus Q}$ from above, we represent $K\backslash Q$ in the form $K_A \bigcup K_B,$
where
$$K_A = \{x+iy \in K:\,  x\in [-1 , \,A-2w]\},\quad K_B = \{x+iy \in K:\, x\in [ B+2w,\, 1]\}.
$$

In the following applications use will be made of Lemma~\ref{lemma6}. All assumptions of the Lemma follow from Subsection \ref{Ss}. Let us verify condition \eqref{l503}. Indeed, because of \eqref{e1} condition \eqref{l503} will be proved once we prove the inequality
$$
w\le \dfrac{(m-3)w}{\sqrt{e^2-1}+2}
\quad \textrm{or equivalently} \quad
\sqrt{e^2-1}+2 \le m-3.
$$
According to the choice \eqref{mcond} of $m$, this holds with equality.

Suppose  $z\in K_A$, then  $|M-z|\le |1-z|$. Applying this,
inequality \eqref{fn1k} from Lemma~\ref{lemma6}, and the estimate
$1+A-2w \ge 1-A-2w \ge (1-\delta)(1-2/m)$
we obtain
\begin{align*}
\frac{|p'(z)|}{n}&=|1+z|^{n-k-1}|1-z|^{k-1}|M-z|\le |1+z|^{n-k-1}|1-z|^k  \le
| f_{n-1,k}(x)|\le
\\&\le
| f_{n-1,k}(A-2w)|= \frac{| p(A-2w+iw)|}{|1+A-2w+iw|}\le  \frac{ | p(A-2w+iw)|}{(1-\delta)(1-2/m)}.
\end{align*}
Consequently,
\begin{equation}\label{estA}
J'_{K_A}=\frac{1}{n^q}\int_{K_A}|p'(z)|^qd\mu(z) \le  \frac{| p(A-2w+iw)|^q}{(1-\delta)^q(1-2/m)^q}\mu(K_A).
\end{equation}

Suppose  $z\in K_B$, then  $|M-z|\le |1+z|.$ Applying this,
inequality \eqref{fn1k1} from Lemma~\ref{lemma6}, and  estimate
\eqref{e3}
we obtain
\begin{align*}
\frac{|p'(z)|}{n}&=|1+z|^{n-k-1}|1-z|^{k-1}|M-z|\le |1+z|^{n-k}|1-z|^{k-1} \le f_{n-1,k-1}(x)\le
\\&\le
| f_{n-1,k-1}(B+2w)|= \frac{| p(B+2w+iw)|}{|1-(B+2w)+iw|}\le
\frac{| p(B+2w+iw)|}{(1-\delta)(1-5/m)}.
\end{align*}
Consequently,
\begin{equation}\label{estB}
J'_{K_B}=\frac{1}{n^q}\int_{K_B}|p'(z)|^qd\mu(z) \le
\frac{| p(B+2w+iw)|^q}{(1-\delta)^q(1-5/m)^q}\mu(K_B).
\end{equation}
Adding \eqref{estA} to \eqref{estB} we get
\begin{align*}
J'_{K\setminus Q}&=J'_{K_A}+J'_{K_B}\le   \frac{| p(A-2w+iw)|^q}{(1-\delta)^q(1-2/m)^q}\mu(K_A)+
\frac{| p(B+2w+iw)|^q}{(1-\delta)^q(1-5/m)^q}\mu(K_B)\le \\
&\le \frac{\mu(K)}{(1-\delta)^q(1-5/m)^q}\max\{| p(A-2w+iw)|,\, | p(B+2w+iw)| \}^q.
\end{align*}
Combining this with \eqref{JQA} and \eqref{JQB} we deduce
\begin{equation}\label{Jp01}
\frac{J'_{K\setminus Q}}{J_Q}\le  \frac{2w^q}{\theta(1-\delta)^q(1-5/m)^q} \max\left\{\frac{| p(A-2w+iw)|}{w^{1+1/q} p(A-w)}, \, \frac{| p(B+2w+iw)|}{w^{1+1/q} p(B+w)}  \right\}^q.
\end{equation}

2.) Now our goal is to find conditions on the degree $n$ of the polynomial $p$ under which
\begin{equation}\label{Jp02}
\dfrac{|p(A-2w+iw)|}{ w^{1+1/q}p(A-w)} \le 1 \quad   \mbox{and} \quad
\dfrac{| p(B+2w+iw)|}{ w^{1+1/q}p(B+w)} \le 1
\end{equation}
or  equivalently
\begin{equation}\label{TwoNeqq}
\ln \frac{|p(A-2w+iw)|}{p(A-w)} <  (1+1/q) \ln w
\end{equation}
and
\begin{equation}\label{TwoNeqq2}
\ln \frac{|p(B+2w+iw)|}{p(B+w)} <  (1+1/q) \ln w.
\end{equation}

We first deal with inequality~\eqref{TwoNeqq}.
To simplify notation we set $A'=A-w$,
$$a := \frac{(1+A'-w)^2+w^2}{(1+A')^2},\quad b := \frac{(1-(A'-w))^2+w^2}{(1-A')^2}=\frac{(1-A'+w)^2+w^2}{(1-A')^2}.$$
Then the left side of \eqref{TwoNeqq} can be written as
\begin{gather*}
\ln \frac{|p(A-2w+iw)|}{p(A-w)}=\ln \frac{|p(A'-w+iw)|}{p(A')}=\frac12\ln a^{n-k}b^{k}=\frac{n}{4}\left((1+M)\ln a + (1-M)\ln b\right).
\end{gather*}
Since $\ln x \leq x-1$ ($x>0$), we have $\ln a \le a-1$ and
$\ln b \leq b-1$.
We proceed
\begin{equation}\label{f02}
\begin{gathered}
 (1+M)\ln a + (1-M)\ln b \leq
(1+M)(a-1) + (1-M)(b-1)  \\
=
(1+M)\frac{-2(1+A')w+2w^2}{(1+A')^2} + (1-M)\frac{2(1-A')w+2w^2}{(1-A')^2}
\\=
\frac{4w}{1-A'^2} \left(-(M-A')+w-\frac{2A'(M-A')w}{1-A'^2}\right)
\\
<
\frac{4w}{1-A'^2} \left(-(M-A')+w\right)=
\frac{4w}{1-A'^2} \left(-(M-A)\right) \le
4w\left(-(M-A)\right) \le
 -4w^2,
 \end{gathered}
\end{equation}
the last inequality following from \eqref{Mcond}.
Therefore, inequality \eqref{TwoNeqq} will be satisfied  once we get
$\dfrac n4(-4w^2)\le (1+1/q) \ln w.$
This is equivalent to the following restriction on $n$:
\begin{equation}\label{Jp04}
    n \ge \frac{(1+1/q)\ln 1/w }{w^2}.
\end{equation}

We now turn to inequality~\eqref{TwoNeqq2}.
With the notation $B':=B+w$,
$$
a_1 := \frac{(1+B'+w)^2+w^2}{(1+B')^2},\quad
b_1 := \frac{(1-(B'+w))^2+w^2}{(1-B')^2}=\frac{(1-B'-w)^2+w^2}{(1-B')^2},
$$
the left side of \eqref{TwoNeqq2} can be written as
\begin{gather*}
\ln \frac{|p(B+2w+iw)|}{p(B+w)}=\ln \frac{|p(B'+w+iw)|}{p(B')}\\
=\frac12\ln a_1^{n-k}b_1^{k}=\frac{n}{4}\left((1+M)\ln a_1 + (1-M)\ln b_1\right).
\end{gather*}

Replacing in \eqref{f02}  $A'$ by $B'=B+w$ and $w$ by $-w$,
and successively applying inequality \eqref{e2} to estimate $1-B'$ from below
and inequality  \eqref{e4} to estimate $B'-M$ from below, we get
\begin{gather*}
 (1+M)\ln a_1 + (1-M)\ln b_1 \le \frac{-4w}{1-B'^2}\left(-(M-B')-w-\frac{2B'(M-B')(-w)}{1-B'^2}\right)
 \\ = \frac{4w}{1-B'^2}\left((M-B')+w+\frac{2B'(B'-M)w}{1-B'^2}\right)\\
\le \frac{4w}{1-B'^2} \left(-(B'-M)+w+\frac{(B'-M)w}{1-B'}\right)\le
\frac{4w}{1-B'^2} \left(-(B'-M)+\frac{B'-M}{m-4}+w\right)
\\
=\frac{-4w^2}{1-B'^2} \left(\dfrac{B'-M}{w}\left(1-\frac{1}{m-4}\right)-1\right)
\le\frac{-4w^2}{1-B'^2} \left(\dfrac52\left(1-\frac{1}{m-4}\right)-1\right)
\\ \le -2w^2\left(3-\dfrac{5}{m-4}\right).
\end{gather*}
The inequality \eqref{TwoNeqq2} will be satisfied  once we get
$$
\dfrac{n}{4}(-2w^2)\left(3-\dfrac{5}{m-4}\right)\le (1+1/q) \ln w.
$$
This is equivalent to
\begin{equation}\label{Jp05}
n \ge \dfrac{2}{\left(3-5/(m-4)\right)}\frac{(1+1/q)\ln 1/w}{w^2}.
\end{equation}
So, we have established \eqref{TwoNeqq2} under restriction \eqref{Jp05}.

Therefore, if both \eqref{Jp04} and \eqref{Jp05} are satisfied, and, moreover, we also have \eqref{nw}, then both \eqref{TwoNeqq} and \eqref{TwoNeqq2} follows, whence \eqref{Jp02} obtains, too.
Combining \eqref{Jp01} and \eqref{Jp02} then yields
\begin{equation}\label{Jp06}
J'_{K\setminus Q}\le  \frac{2}{\theta(1-\delta)^q(1-5/m)^q}w^q J_Q
\end{equation}
for $n$  satisfying \eqref{Jp04}, \eqref{Jp05} and \eqref{nw}.

\subsection{Estimate of $J'_K$}
Suppose $0<q<\infty.$ Estimates \eqref{JpQ} and \eqref{Jp06} give us the numbers $\alpha$ and $\beta$ from  \eqref{ab}.
Substituting them into \eqref{abf}
  we obtain
 \begin{align*}
C_{q}(\delta,\theta)&\le 4(\alpha+\beta)^{1/q} =
4\left(  \dfrac{17^{q/2}}{(1-\delta)^q(1-5/m)^q}+
\frac{2}{\theta(1-\delta)^q(1-5/m)^q}  \right)^{1/q}
\\&\le
\frac{4\cdot 17^{1/2}}{(1-5/m)}\dfrac{1}{(1-\delta)}\left( 1 +\frac{2}{\theta}  \right)^{1/q}
\end{align*}
for $n$  satisfying \eqref{Jp04}, \eqref{Jp05} and \eqref{nw}.

Let us compare these three conditions.
Observe that $7<m<8$, which entails
\begin{equation*}\label{ncond01}
1<\dfrac{2}{3-5/4}<\dfrac{2}{3-5/(m-4)} <
\dfrac{2}{3-5/3}=1.5.
\end{equation*}
Thus we can replace the restriction \eqref{Jp05}  by
\begin{equation}\label{ncond1}
n > \frac{1.5(1+1/q)\ln 1/w}{w^2}.
\end{equation}
which is  stronger than \eqref{Jp04} and \eqref{ncondTh} in Theorem~\ref{theorem1}.
From   $w\le(1-\delta)/m\le 1/7$ and monotonicity of $x\ln x$ for $x\ge 7$ we get that $(1/w)\ln(1/w) \ge 7\ln 7\ge 7$. Thus it follows that
$$
\dfrac{2}{\left(3-5/(m-4)\right)}\frac{(1+1/q)\ln 1/w}{w^2}\ge \frac{\ln 1/w}{w^2}\ge \dfrac{7}{w}.
$$
Thus  \eqref{Jp05} is stronger than \eqref{nw} for all $q$, and in fact even without the term containing $q$.

Finally  taking into account that $4 \cdot 17^{1/2}/ (1 - 5/m)\le 4 \cdot 5/ (1- 5/7)=
70<121$ we obtain the assertion of Theorem~\ref{theorem1} for $0<q<\infty$ and all $n$ satisfying \eqref{ncond1}.

We are left with the case $q=\infty.$
Observe that the constructed polynomial  $p$ is independent of $q$,
and that any $n\ge \dfrac{2\ln 1/w}{w^2}$ meets \eqref{ncond1} for  sufficiently large $q$.
Therefore we can pass to the limit (see, e.g., \cite[Problem 4.7.44]{Bogachev}) for all $n \ge \dfrac{2\ln 1/w}{w^2}$
and obtain
\begin{equation}\label{qinfty}
M_{n,\infty}(K,\mu)\le \dfrac{\|p'\|_\infty}{\|p\|_\infty}=\lim_{q\to \infty}\dfrac{\|p'\|_q}{\|p\|_q}\le
\limsup _{q\to \infty}C_q(\delta,\theta)\dfrac{w}{d^2}n\le \dfrac{121}{1-\delta}\dfrac{w}{d^2}n.
\end{equation}
This proves Theorem~\ref{theorem1} for $q=\infty.$

\section{The case of a large width}

We now turn to the case $(1-\delta)/m < w\le 1$. Recall that  there exist $\theta \in(0,1)$ and $\delta \in (0,1)$ such that \eqref{ogrMu} holds.
Then at least one of the ``halves''
$K_\ell=K_\delta \cap( [-\delta,0]\times [-iw,iw])$ or
$K_r=K_\delta \cap ([0,\delta]\times [-iw,iw])$ has measure greater than $\dfrac \theta 2 \mu(K).$ Without loss of generality we assume that $\mu(K_\ell)\ge \dfrac \theta 2 \mu(K).$

Let us take the polynomial   $p(z)=(1-z)^n$, $n\ge 2.$

As in Section~\ref{Secsmallw} we first deal with $0 < q <\infty$.
We divide $K$ into two subsets
$$K_1 = \{x+iy \in K \ : \ x\in [-1, 3/4]\},\quad \mbox{and} \quad
K_2 = \{x+iy \in K \ : \ x\in [3/4, 1]\}.$$
The estimate $|1-z|\geq |1-x|\ge 1/4$ for $z\in K_1$ yields
\begin{equation}\label{bw02}
\left\|p'\right\|_{L^q(K_1,\mu)}^q=n\left\|(1-z)^{n-1}\right\|_{L^q(K_1,\mu)}^q \le n^q
\max_{z\in K_1}\frac{1}{|1-z|^q}\|p\|_{L^q(K_1,\mu)}^q \le  (4n)^q\|p\|_q^q .
\end{equation}
We proceed to estimate $\left\|p'\right\|_{L^q(K_2)}.$
Since  the diameter $d(K)=2$ and $-1, 1\in K$, the set $K$ is bounded by
the circle of radius $2$ and centered at the point $-1$, thus $y^2\le 4-(1+x)^2$ for any $z=x+iy\in K$.
This gives
$$|1-z|^2=(1-x)^2+y^2 \le (1-x)^2+4-(1+x)^2= 4-4x\le 1, \quad z\in K_2,$$
hence
\begin{equation}\label{bw01}
\left\|p'\right\|_{L^q(K_2,\mu)}^q=n^q\left\|(1-z)^{n-1}\right\|_{L^q(K_2,\mu)}  \leq
n^q\max_{z\in K_2}|1-z|^{(n-1)q}\mu(K_2) \leq n^q \mu(K).
\end{equation}
To  estimate $\mu(K)$ we note that
$\ds
\|p\|_q^q \ge \|p\|_{L^q(K_\ell)}^q \ge \min_{z\in K_\ell}|1-z|^{nq} \mu(K_\ell) \ge \dfrac{\theta}{2} \mu(K),
$
thus $\mu(K)\le \dfrac 2\theta \|p\|_q^q.$
Substituting this into \eqref{bw01} yields
\begin{equation}\label{bw03}
\left\|p'\right\|_{L^q(K_2,\mu)}^q \le n^q  \frac 2\theta  \|p\|_q^q.
\end{equation}
Adding \eqref{bw02} to \eqref{bw03} we deduce that
$$
\|p'\|_q^q \le \left((4n)^q +n^q\frac{2}{\theta}  \right)\|p\|_q^q \le
(4n)^q\left(1 +\frac{2}{\theta}  \right) \|p\|_q^q,
$$
and finally that
$$\dfrac{\|p'\|_q}{\|p\|_q} \le 4n\left(1 +\frac{2}{\theta}  \right)^{1/q} \le
4n\frac{m}{1-\delta}w \dfrac{4}{d^2}\left(1 +\frac{2}{\theta}  \right)^{1/q}\le
\frac{16m}{1-\delta}\left(1 +\frac{2}{\theta}  \right)^{1/q}\frac{w}{d^2}n.
$$
It remains to observe that $16m\le 121.$

A passage to the limit similar to \eqref{qinfty} proves the case $q=\infty$.

\section{Proof of the Corollaries}

Suppose first that $\mu$ is the linear Lebesgue measure on the boundary of $K$. We want to derive  Corollary \ref{cor1} from Theorem \ref{theorem1}.

\begin{proof}[Proof of Corollary \ref{cor1}] It is plausible that to any $0<\de<1$ there is some appropriate $\theta$ with which Theorem \ref{theorem1} can be applied. To obtain an effective estimate we first look for a concrete admissible value of $\theta$ corresponding to an arbitrary value of $\de\in(0,1)$.

 By Lemma~\ref{lemma1},
 $K$ is contained in the rectangle $[-1,1]\times [-iw, iw]$.
 By properties of convex curves \cite[p. 52, Property 5]{Bonnesen}
the arc length of $\partial K$ is not greater than the perimeter of the rectangle, i.e. $\mu(K)\le 4(1+w)\le 8$, while
the arc length of the part of $\partial K$ belonging to $K_\delta$ is trivially at least $4\delta$, i.e. $\mu(K_\delta)\ge 4\delta$.
This provides us with the admissible value $\theta=\de/2.$

Therefore, an application of Theorem \ref{theorem1} leads to
\begin{equation}\label{C1plus}
M_{n,q}(K) \le C_q(\delta) \frac{w^2}{d} n, \quad \textrm{where} \quad C_q(\delta):=C_q\left(\delta,\frac{\delta}2\right)=
\frac{121}{1-\delta}\left(1 +\frac{4}{\delta}  \right)^{1/q} \qquad (0<\de<1).
\end{equation}
 It remains to optimize on the choice of $\delta$ in function of $q$. So now we are looking for $\ds \min_{0<\de<1} \dfrac{1}{121} C_q(\de) =\min_{0<\de<1} \dfrac{(1+4/\de)^{1/q}}{1-\de}$. Denote
 $$\ff(\de):= \log \dfrac{(1+4/\de)^{1/q}}{1-\de} = \dfrac{1}{q}\log(1+4/\de)-\log(1-\de).$$
 It is easy to see that on $(0,1)$ this function is strictly convex, with $\ff(0)=\ff(1)=+\infty$ determining a $U$-shape form. We are looking for its unique minimum point. Differentiation results in
$\ff'(\de)= \dfrac{-4}{q} \dfrac{1}{\de^2+4\de} +\dfrac{1}{1-\de}$, and the unique root of this in $(0,1)$ can be obtained from the quadratic equation $\de^2+4\de=\dfrac{4}{q}(1-\de)$ or $\de^2+(4+4/q)\de-4/q=0$. This equation has two roots, one in the negative semiaxis and one in $(0,1)$: the latter is $\de_1:= \frac{2}{q} \left(Q-(q+1) \right)$, where $Q:=\sqrt{q^2+3q+1}$. This latter point is the unique minimum point for $\ff(\de)$, so that we get
\begin{align*}
\min_{0<\de<1} \frac{1}{121} C_q(\de) &=\min_{0<\de<1} \exp(\ff(\de)) = \exp(\ff(\de_1))= \frac{\{1+4/(\frac{2}{q}(Q-(q+1)))\}^{1/q}}{1-\frac{2}{q}(Q-(q+1))}
\\&= q \frac{\{1+2q\frac{Q+q+1}{Q^2-(q+1)^2}\}^{1/q}}{3q+2-2Q}
  =q (3q+2+2Q) \frac{\{1+2q\frac{Q+q+1}{q}\}^{1/q}}{(3q+2)^2-4Q^2}
\\&= q (3q+2+2Q) \frac{\{1+2(Q+q+1)\}^{1/q}}{5q^2}
= \frac{3q+2+2Q}{5q} (3+2q+2Q)^{1/q}.
\end{align*}
Substituting back into \eqref{C1plus} yields the stated inequality \eqref{C1}.
\end{proof}

Next, we consider the case when $\mu$ is the (restriction to $K$) of the 2-dimensional Lebesgue measure (area).

\begin{proof}[Proof of Corollary \ref{cor2}]
Again by Lemma~\ref{lemma1},  $\mu (K) \le \mu ([-1,1]\times [-iw, iw])= 4w.$

Taking into account the definition of the minimal width, we observe that there exist
$\alpha \in [0,1]$ and points $x^+$, $x^-\in [-1,1]$ such that the points $z^+=x^++i\alpha w$ and $z^-=x^--i(1-\alpha )w$ belong to $K$ (as otherwise the width of $K$ in the vertical direction would be less than $w$).
Write $K^+_\delta=\{x+iy\in K_\delta \colon  y\ge 0\},$ $K^-_\delta=\{x+iy\in K_\delta \colon  y\le 0\}.$

If $x^+\in[-\delta,\delta]$, then by convexity of $K$, the value of $\mu(K_\delta^+)$ is at least the area of the triangle with the vertices $-\delta$, $z^+$, $\delta$, which is equal to $\alpha w \delta$.
If $x^+\in[\delta,1]$, then $\mu(K_\delta^+)$ is not less then the area of the trapezium bounded by the real axis, the straight lines $x=\pm\delta$, and the straight line passing through the points $-1$ and $1+i\alpha w.$ The area of this  trapezium is also   $\alpha w \delta$. Similarly, if $x^+\in[-1,-\delta]$,   then $\mu(K_\delta^+)$ is not less then the area of the  trapezium bounded by the real  axis, the straight lines $x=\pm\delta$, and the straight line passing through the points $-1+i\alpha w$ and $1$.
Hence, in all cases, $\mu(K_\delta^+)\ge \alpha w \delta.$

Analogously, $\mu(K_\delta^-)\ge (1-\alpha) w \delta.$
Finally, $\mu(K_\delta) \ge w\delta$. Therefore, $\mu(K_\delta)/\mu(K)\ge \delta/4$.

Thus for any $0<\de<1$ we can set $\theta:=\delta/4$ as an admissible value for an application of Theorem~\ref{theorem1}. The theorem provides
\begin{equation}\label{C2plus}
M_{n,q}(K,\lambda) \le C_q^*(\delta) \frac{w^2}{d} n, \quad \textrm{where} \quad C_q^*(\delta):=C_q\left(\delta,\frac{\delta}4\right)
=\frac{121}{1-\delta}\left(1 +\frac{8}{\delta}  \right)^{1/q}.
\end{equation}
Now we consider $\ff(\de):=\log \frac{1}{121} C_q^*(\de)$, which is again a strictly convex, $U$-shaped function with endpoint values $\ff(0)=\ff(1)=+\infty$. To find the minimum, we compute the unique critical point satisfying $\ff'(\de)=0$. This equation can be written as
$$
\frac1q \frac{1}{1+8/\de}\cdot(-\frac{8}{\de^2})-\frac{1}{1-\de}\cdot(-1)=0 \quad \text{or} \quad q\de^2+(8q+8)\de-8=0.
$$
The quadratic equation has two roots, one in the negative semiaxes and another one in $(0,1)$: this latter one is $\de_1=\frac2q \left(Q-2q-2\right)$, where $Q:=\sqrt{4q^2+10q+4}$. This is the minimum point of $\ff(\de)$. Therefore,
\begin{align*}
\min_{0<\de<1} \frac{1}{121} C_q^*(\de) &=\min_{0<\de<1} \exp(\ff(\de)) = \exp(\ff(\de_1))= \frac{\{1+8/(\frac{2}{q}(Q-2q-2))\}^{1/q}}{1-\frac{2}{q}(Q-2q-2))}
\\&= q \frac{\left(1+\frac{4q}{Q-2q-2}\right)^{1/q}}{q-2Q+4q+4}
  =\frac{q(5q+4+2Q)}{(5q+4)^2-4Q^2} \left(1+ \frac{4q(Q+2q+2)}{Q^2-(2q+2)^2} \right)^{1/q}
\\&= \frac{q(5q+4+2Q)}{9q^2} \left(
1+\frac{4q(Q+2q+2)}{2q}\right)^{1/q}
= \frac{5q+4+2Q}{9q}
\left(4q+5+2Q\right)^{1/q}.
\end{align*}
Substitution in \eqref{C2plus} results in the stated assertion \eqref{C2}.
\end{proof}

\newcommand{\No}{N}

Our present work has dealt with the case of complex \emph{domains} with positive width $w>0$.
The partition of the set $K$ and the calculus depended on $w$ and are valid only for $n>d^2/w^2$, a~fatal condition if $w=0$. 

As we mention in the introduction, in all known cases the \emph{lower estimates} for $w=0$ have $\sqrt{n}$ order. However, we are not aware of general \emph{upper estimations} which would ensure that the order is indeed $\sqrt{n}$. The most general we are aware of is the remark of Xiao and Zhou following \cite[p. 198, Theorem~1, Corollary 1]{Xiao} and pointing out that their estimate is of the right order for Jacobi weights. It seems likely that $O(\sqrt{n})$ upper estimates for $M_{n,q}(I,\mu)$ remain valid even for more general weights and measures.


\label{lastpage}


{\it Acknowledgments.}
P. Yu. Glazyrina was supported by  the Ministry of Science and Higher Education of the Russian Federation (Ural Federal University Program of Development within the Priority-2030 Program).

Sz. Gy.~R\'{e}v\'{e}sz was supported in part by Hungarian National Research,
Development and Innovation Fund, project \# \# K-132097.

\bigskip

\noindent
\hspace*{5mm}
\begin{minipage}{\textwidth}
\noindent
\hspace*{-5mm}
Polina Yu.{} Glazyrina\\
Institute of Natural Sciences and Mathematics,\\ Ural Federal University,\\
 Mira street 19\\
 620002 Ekaterinburg, Russia\\
\end{minipage}

\bigskip

\noindent
\hspace*{5mm}
\begin{minipage}{\textwidth}
\noindent
\hspace*{-5mm}
Yulia S.{} Goryacheva\\
Institute of Natural Sciences and Mathematics,\\ Ural Federal University,\\
 Mira street 19\\
 620002 Ekaterinburg, Russia\\
\end{minipage}

\bigskip

\noindent
\hspace*{5mm}
\begin{minipage}{\textwidth}
\noindent
\hspace*{-5mm}
Szilárd Gy.{} Révész\\
 Alfréd Rényi Institute of Mathematics\\
 Reáltanoda utca 13-15\\
 1053 Budapest, Hungary \\
\end{minipage}

\end{document}